\documentclass[10pt]{article}
\usepackage{graphics,psfrag}
\usepackage{latexsym}
\usepackage{graphicx}
\usepackage[sf, bf, small]{titlesec}
\usepackage{relsize}
\catcode`@=11
\renewcommand{\@begintheorem}[2]{\it \trivlist
      \item[\hskip \labelsep{\bf #1\ #2{\rm :}}]}
\renewcommand{\@opargbegintheorem}[3]{\it \trivlist
      \item[\hskip \labelsep{\bf #1\ #2\ {\rm (#3)\/:}}]}
\def\@sect#1#2#3#4#5#6[#7]#8{\ifnum #2>\c@secnumdepth
     \def\@svsec{}\else
     \refstepcounter{#1}\edef\@svsec{\csname the#1\endcsname{.}\hskip 1em }\fi
     \@tempskipa #5\relax
      \ifdim \@tempskipa>\z@
        \begingroup #6\relax
          \@hangfrom{\hskip #3\relax\@svsec}{\interlinepenalty \@M #8\par}
        \endgroup
       \csname #1mark\endcsname{#7}\addcontentsline
         {toc}{#1}{\ifnum #2>\c@secnumdepth \else
                      \protect\numberline{\csname the#1\endcsname}\fi
                    #7}\else
        \def\@svsechd{#6\hskip #3\@svsec #8\csname #1mark\endcsname
                      {#7}\addcontentsline
                           {toc}{#1}{\ifnum #2>\c@secnumdepth \else
                             \protect\numberline{\csname the#1\endcsname}\fi
                       #7}}\fi
     \@xsect{#5}}

\@addtoreset{equation}{section}

\newcommand{\Delete}[1]{}

\usepackage{amsmath,amssymb,amsthm}

\theoremstyle{plain}
\newtheorem{Thm}{Theorem}[section]

\newcommand{\cF}{\ensuremath{\mathcal{F}}}

\title{Cycles and $p$-competition graphs}

\author{
\begin{tabular}{c}
{\sc Suh-Ryung KIM}
\thanks{This work was supported by the Korea Research Foundation Grant
funded by the Korean Government (MOEHRD) (KRF-2008-531-C00004).}
\quad
{\sc Boram PARK}
\thanks{The author was supported by Seoul Fellowship.}
\thanks{Corresponding author. {\it E-mail address:}
kawa22@snu.ac.kr ; borampark22@gmail.com}
\\
[1ex]
Department of Mathematics Education, \\
Seoul National University, Seoul 151-742, Korea. \\
\\
{\sc Yoshio SANO}
\thanks{The author was supported by JSPS Research Fellowships
for Young Scientists.
The author was also supported partly by Global COE program
``Fostering Top Leaders in Mathematics".}\\
[1ex]
Research Institute for Mathematical Sciences, \\
Kyoto University, Kyoto 606-8502, Japan. \\
\end{tabular}  }

\date{May 2009}

\begin{document}

\maketitle

\begin{abstract}
The notion of $p$-competition graphs of digraphs
was introduced by
S-R.\ Kim, T.\ A.\ McKee, F.\ R.\ McMorris, and F.\ S.\ Roberts
[$p$-competition graphs,
{\it Linear Algebra Appl.} {\bf 217} (1995) 167--178]
as a generalization of the competition graphs of digraphs.
Let $p$ be a positive integer.
The {\it $p$-competition graph} $C_p(D)$ of a digraph $D=(V,A)$ is
a (simple undirected) graph
which has the same vertex set $V$
and has an edge between distinct vertices $x$ and $y$
if and only if there exist $p$ distinct vertices
$v_1, \ldots, v_p \in V$ such that
$(x,v_i), (y,v_i)$ are arcs of the digraph $D$ for each $i=1, \ldots, p$.

In this paper, given a cycle of length $n$,
we compute exact values of $p$ in terms of $n$ such that
it is a $p$-competition graph, which generalizes
the results obtained by Kim~{\it et al.}.
We also find values of $p$ in terms of $n$ so that its complement
is a $p$-competition graph.

\end{abstract}

\vspace{4mm}

\newpage
\section{Introduction}

Cohen \cite{Cohen1} introduced the notion of
a competition graph in connection with a problem in ecology
(also see \cite{Cohen2}).
Let $D=(V,A)$ be a digraph that corresponds to a food web.
A vertex in the digraph $D$ stands for a species in the food web,
and an arc $(x,v) \in A$ in $D$ means that the species $x$
preys on the species $v$.
For vertices $x, v \in V$,
we call $v$ a {\it prey} of $x$ if there is an arc $(x,v) \in A$.
If two species $x$ and $y$ have a common prey $v$,
they will compete for the prey $v$.
Cohen defined a graph
which represents the competition among the species
in the food web.
The {\it competition graph} $C(D)$ of a digraph $D=(V,A)$
is a simple undirected graph $G=(V,E)$ which
has the same vertex set $V$
and has an edge between two distinct vertices
$x, y \in V$ if
there exists a vertex $v \in V$
such that $(x,v), (y,v) \in A$.
We say that a graph $G$ is a {\it competition graph}
if there exists a digraph $D$ such that $C(D)=G$.

Kim {\it et al.}~\cite{MR1322549}
introduced $p$-competition graphs
as a generalization of competition graphs.
Let $p$ be a positive integer.
The {\it $p$-competition graph} $C_p(D)$ of a digraph $D=(V,A)$
is a simple undirected graph $G=(V,E)$
which has same vertex set $V$ and has an edge between two distinct vertices
$x, y \in V$
if there exist $p$ distinct vertices $v_1, ..., v_p \in V$
such that $(x,v_i), (y,v_i) \in A$
for each $i=1, ..., p$.
Note that the $1$-competition graph $C_1(D)$ of a digraph $D$
coincides with the competition graph $C(D)$ of the digraph $D$.
A graph $G$ is called a {\it $p$-competition graph}
if there exists a digraph $D$ such that $C_p(D) =G$.

Competition graphs are closely related to
edge clique covers and the edge clique cover numbers of graphs.
A {\it clique} of a graph $G$ is a subset of
the vertex set of $G$ such that
its induced subgraph of $G$ is a complete graph.
We regard an empty set also as a clique of $G$ for convenience.
An {\it edge clique cover}
of a graph $G$ is a family of cliques of $G$
such that the endpoints of each edge of $G$
are contained in some clique in the family.
The minimum size of an edge clique cover of $G$ is called
the {\it edge clique cover number}
of the graph $G$, and is denoted by $\theta_e(G)$.

Let $G=(V,E)$ be a graph and $\cF=\{S_1, ..., S_r\}$
be a multifamily of subsets of the vertex set of $G$.
The multifamily $\cF$ is called a {\it $p$-edge clique cover}
if it satisfies the following:
\begin{itemize}
\item
For any $J \in {[r] \choose p}$, the set
$\bigcap_{j \in J} S_j$ is a clique of $G$,
\item
The family $\{\bigcap_{j \in J} S_j \mid J \in {[r] \choose p} \}$
is an edge clique cover of $G$,
\end{itemize}
where ${[r] \choose p}$ denotes the set of all $p$-subsets
of an $r$-set $[r]:=\{1,2,...,r\}$.
The minimum size of a $p$-edge clique cover of $G$
is called the {\it $p$-edge clique cover number} of $G$,
and is denoted by $\theta_e^p(G)$.

Kim {\it et al.}~\cite{MR1322549} gave a characterization of
$p$-competition graphs, a sufficient condition
for graphs to be $p$-competition graphs,
and a characterization of the $p$-competition graphs of
acyclic digraphs as follows:

\begin{Thm}[\cite{MR1322549}]\label{thm:chara}
Let $G$ be a graph with $n$ vertices.
Then $G$ is a $p$-competition graph
if and only if $\theta_e^p(G) \leq n$.
\end{Thm}

\begin{Thm}[\cite{MR1322549}] \label{thm:Cor2}
Let $G$ be a graph with $n$ vertices.
If $\theta_e(G) \leq n-p+1$, then $G$ is a $p$-competition graph.
\end{Thm}

\begin{Thm} [\cite{MR1322549}] \label{thm:chara-acyclic}
Let $G=(V,E)$ be a graph with $n$ vertices.
Then, $G$ is the $p$-competition graph of an acyclic digraph
if and only if
there exist an ordering $v_1, ..., v_n$ of the vertices of $G$
and a $p$-edge clique cover $\{S_1, ..., S_n \}$ of $G$
such that $v_i \in S_j$ $\Rightarrow$ $i < j$.
\end{Thm}

In this paper, given a cycle of length $n$, we compute exact values of $p$ in terms of $n$ such that it is a $p$-competition graph, which generalizes the results obtained by Kim~{\it et al.}~\cite{MR1322549}.  We also find values of $p$ in terms of $n$ so that its complement is a $p$-competition graph.
Throughout this paper,
let $C_n$ be a cycle with $n$ vertices:
\[
V(C_n) = \{v_0, ..., v_{n-1} \}, \quad
E(C_n) = \{v_i v_{i+1} \mid i=0, ..., n-1 \},
\]
and any subscripts are reduced to modulo $n$.



\section{Main Results}


Kim {\it et al.}~\cite{MR1322549} showed the following results:

\begin{Thm}[\cite{MR1322549}]\label{thm:C4}
$C_4$ is not a $2$-competition graph.
\end{Thm}

\begin{Thm}[\cite{MR1322549}]\label{thm:Cn2-comp}
For $n>4$, $C_n$ is a $2$-competition graph.
\end{Thm}

The following theorem is a generalization of the above results.

\begin{Thm}
Let $C_n$ be a cycle with $n$ vertices and $p$ be a positive integer.
Then $C_n$ is a $p$-competition graph if and only if $n \geq p+3$.
\end{Thm}

\begin{proof}
First, we show the `only if' part by contradiction.
Suppose that the cycle $C_n$ is a $p$-competition graph with
a positive integer $p$ such that $p \geq n-2$.
By Theorem \ref{thm:chara}, there exists
a $p$-edge clique cover $\cF$ of $C_n$
which consists of $n$ sets.

Suppose that $v_i$
belongs to all the $n$ sets in $\cF$ for some $i \in \{0,1, \ldots, n-1\}$.
Then, since  $v_{i+1}$ and $v_{i+2}$ are adjacent,
there exist $p$ sets in $\cF$ which contain both
$v_{i+1}$ and $v_{i+2}$.
Since those $p$ sets contain $v_i$,
the vertices $v_{i}$ and $v_{i+2}$ are adjacent, which is a contradiction.
Thus, each vertex on $C_n$ belongs to at most $n-1$ sets in $\cF$.
Since $v_i$ is adjacent to $v_{i+1}$ and $v_{i-1}$,
it must belong to at least
$p$ sets which contain $v_{i+1}$
and at least $p$ sets which contain $v_{i-1}$.
On the other hand, since $v_{i+1}$ and $v_{i-1}$ are not adjacent,
there are at most $p-1$ sets
which contain both $v_{i+1}$ and $v_{i-1}$.
Thus $v_i$ belongs to at least
$(p+p)-(p-1)=p+1$
sets in $\cF$.
Therefore
\[
p+1 \le (\text{the number of sets in }\cF\text{ containing }v_i ) \le n-1.
\]
However, $p+1 \geq n-1$ by our assumption that $p \geq n-2$.
Thus $v_i$ belongs to exactly $n-1$ sets in $\cF$.
Since $v_i$ was arbitrarily chosen, $v_{i+2}$ also belongs
to exactly $n-1$ sets in $\cF$.
Then $v_i$ and $v_{i+2}$ belong to at least $n-2$ sets together.
However, since $n-2=p$, $v_i$ and $v_{i+2}$ are adjacent
and we reach a contradiction. Hence $n \geq p+3$.

Now we show the `if' part.
Let $p$ be a positive integer such that $p+3 \leq n$.
Let
\[
S_i:=\{v_i, v_{i+1}, ..., v_{i+p} \} \quad (i=0, ..., n-1).
\]
Then we put $\cF:=\{S_0, S_1, \ldots, S_{n-1}\}$.
It is easy to check that for each $i$, $j$  with $i < j$,
\[
S_i, S_{i-1}, \ldots, S_{j-p}
\]
are the sets containing both $v_i$ and $v_j$.
Thus there are $p$ sets in $\cF$ containing both $v_i$ and $v_{i+1}$,
and
there are at most $p-1$ sets in $\cF$
containing both $v_i$ and $v_j$
for $i$, $j$ with $j-i \ge 2$.
Thus $\cF$ is a $p$-edge clique cover of $C_n$.
Hence $\theta_e^p(C_n) \leq n$.
By Theorem \ref{thm:chara},
the cycle $C_n$ is a $p$-competition graph.
\end{proof}



In the rest of the paper,
we give a sufficient condition for
the complements of cycles to be $p$-competition graphs.
We first study on the edge clique cover numbers of complements of cycles:

\begin{Thm}\label{lem:ECC}
Let $\overline{C_n}$ be the complement of a cycle with $n$ vertices.
Then the following are true:
\begin{itemize}
\item[{\rm (i)}]
$\theta_e(\overline{C_5})=5$,
$\theta_e(\overline{C_6})=5$,
$\theta_e(\overline{C_7})=7$,
$\theta_e(\overline{C_8})=6$.
\item[{\rm (ii)}]
If $n$ ($n \geq 9$) is odd,
then $\theta_e(\overline{C_n}) \le \frac{n+5}{2}$.
\item[{\rm (iii)}]
If $n$ ($n \geq 10$) is even,
then $\theta_e(\overline{C_n}) \le \frac{n}{2}+1$.
\end{itemize}
\end{Thm}

\begin{proof}
(i)
Since $\overline{C_5} \cong C_5$, we have $\theta_e(\overline{C_5})=5$.

For $\overline{C_6}$,
any two edges of $\{v_0v_3, v_1v_4, v_2v_5, v_1v_5, v_2v_4 \}$
cannot be covered by a clique as the set of their endpoints
contains two consecutive vertices of $C_6$.
Thus $\theta_e(\overline{C_6})\ge 5$.
In addition,
$\{v_0,v_2,v_4 \}, \{ v_1,v_3,v_5\}, \{ v_2,v_5\}, \{v_1,v_4\},
\{v_0,v_3\}$ form an edge clique cover of $\overline{C_6}$,
and so $\theta_e(\overline{C_6})=5$.

For $\overline{C_7}$,
any two edges of
$\{ v_0v_3, v_1v_4, v_2v_5, v_3v_6, v_4v_0, v_5v_1, v_6v_2 \}$
cannot be covered by a clique
and so $\theta_e(\overline{C_7})\ge 7$.
Since
\[
\{v_0,v_2,v_5 \}, \{ v_1,v_3,v_6\}, \{ v_2,v_0,v_4\}, \{v_3,v_1,v_5\},
\{v_4,v_2,v_6  \}, \{ v_0,v_3 \}, \{v_1,v_4 \}
\]
form an edge clique cover of $\overline{C_7}$,
 $\theta_e(\overline{C_7})=7$.

\begin{figure}
\begin{center}
\psfrag{a}{$v_0$}
\psfrag{b}{$v_1$}
\psfrag{c}{$v_2$}
\psfrag{d}{$v_3$}
\psfrag{e}{$v_4$}
\psfrag{f}{$v_5$}
\psfrag{g}{$v_6$}
\psfrag{h}{$v_7$}
 { \includegraphics[width=200pt]{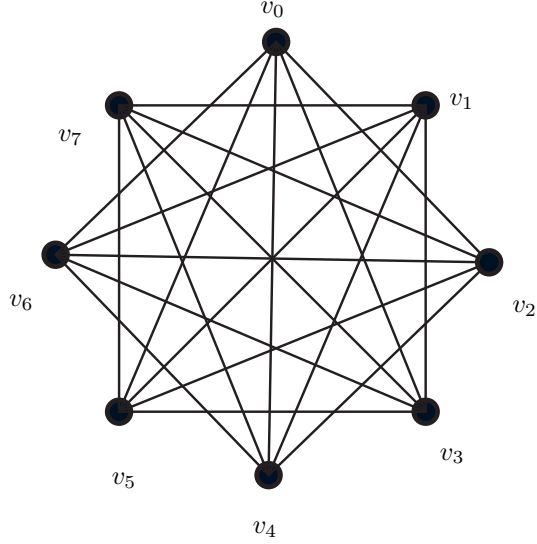} }\\
  \caption{ $\overline{C_8}$}\label{C8bar}
  \end{center}
\end{figure}

For $\overline{C_8}$,
any two edges of $\{v_0v_2, v_0v_3, v_1v_3, v_1v_4, v_2v_7 \}$
cannot be covered by same clique and so $\theta_e(\overline{C_8})\ge 5$
(see Figure~\ref{C8bar} for illustration).
Suppose that there is an edge clique cover of $\overline{C_8}$
whose size is $5$, say $\{S_1,S_2,S_3,S_4,S_5\}$.
By the definition of an edge clique cover,
\[20=|E(\overline{C_8})| \leq \sum_{i=1}^{5} |E(S_i)|.\]
Without loss of generality,
we may assume that $|S_i|\ge |S_{i+1}|$ for $i=1,2,3,4$.
Since $\{v_0,v_2,v_4,v_6\}$ and $\{v_1,v_3,v_5,v_7\}$ are the only cliques
of maximum size, we have $|S_i|\le 3$ for $i=3,4,5$.
If $|S_2|\le 3$, then
\[\sum_{i=1}^{5} |E(S_i)| \le 6 + 12 =18,\]
which is a contradiction.
Thus $|S_1|=|S_2|=4$.
Then we may assume that $S_1=\{v_0,v_2,v_4,v_6\}$ and
$S_2=\{ v_1,v_3,v_5,v_7 \}$.
No two of edges $v_0v_3,  v_1v_4, v_2v_7$ belong to the
same clique. Moreover none of them is covered by $S_1$ or $S_2$.
Thus they must be
covered by $S_3$, $S_4$, $S_5$.
Without loss of generality, we may assume that $\{v_0,v_3\} \subseteq S_3$,
$\{v_1,v_4\} \subseteq S_4$ and $\{v_2,v_7\} \subseteq S_5$.
The edges $v_3v_6$, $v_1v_6$, $v_0v_5$, $v_2v_5$ still need
to be covered.
The edge $v_1v_6$ cannot be covered by $S_3$ or $S_5$
and so must be covered by $S_4$.  Thus $S_4=\{v_1, v_4, v_6\}$.
Now the edge $v_3v_6$ must be covered by $S_3$
and so $S_3=\{v_0,v_3,v_6\}$.  Then $v_0v_5$
cannot be covered by any of $S_1$, \ldots, $S_5$, which is a contradiction.
Thus $\theta_e(\overline{C_8})\ge 6$.
Furthermore,
\[
\{v_0,v_3,v_5 \}, \{ v_2,v_5,v_7\}, \{  v_4,v_1,v_7\},
\{v_6,v_1,v_3\}, \{v_0,v_2,v_4,v_6 \}, \{ v_1,v_3,v_5,v_7 \}
\]
form an edge clique cover of $\overline{C_8}$,
and so $\theta_e(\overline{C_8})=6$.

\noindent
(ii)
Let $n$ ($n \geq 9$) be an odd integer.
We define the following sets:
\begin{align*}
S_1&:=\{ v_0, v_2, v_4, \ldots, v_{n-3} \}; \\
S_2&:=\{ v_0, v_3, v_5, \ldots, v_{n-2} \}; \\
S_3&:=\{ v_1, v_3, v_5, \ldots, v_{n-2} \}; \\
T_1&:=\{ v_1, v_4, v_6, \ldots, v_{n-1} \};   \\
T_3&:=\{ v_3, v_6, v_8, \ldots, v_{n-1} \};     \\
T_5&:=\{ v_5, v_8, v_{10}, \ldots, v_{n-1}, v_2 \};     \\
T_i&:=\{v_i,v_{2}, v_4, \ldots, v_{i-3}, v_{i+3},v_{i+5},
\ldots, v_{n-3}, v_{n-1} \} \quad (i= 7, \ldots, n-2, i: \text{ odd }).
\end{align*}
It is easy to check that each of
$S_i$ $(i=1, 2, 3)$ and  $T_j$ $(j=1, 3, \ldots, n-2)$
is a clique of $\overline{C_n}$.
Take any edge $v_iv_j$ of $\overline{C_n}$ with $i<j$.
We consider all the possible cases in the following:
\begin{itemize}
\item{}
If both $i$ and $j$ are odd, then $v_iv_j$ is covered by $S_3$.
\item{}
If $i$ is odd and $j$ is even,
then $j-i \ge 3$ and $v_iv_j$ is covered by $T_i$.
\item{}
Suppose that $i$ is even and $j$ is odd. Then $j-i \ge 3$.
If $i=0$, then $v_iv_j$ is covered by $S_2$.
If $i \neq 0$, then $v_iv_j$ is covered by $T_j$.
\item{}
Suppose that both $i$ and $j$ are even.
If $j \neq n-1$ then $v_iv_j$ is covered by $S_1$.
If $j=n-1$, then $i \neq 1$ and $v_iv_j$ is covered by  $T_1$ or $T_5$.
\end{itemize}
Therefore the family $\{ S_1, S_2, S_3, T_1, T_3, T_5, \ldots, T_{n-2}\}$
is an edge clique cover of $\overline{C_n}$,
and so $\theta_e(\overline{C_n}) \le 3+ \frac{n-1}{2}=\frac{n+5}{2}$.

\noindent
(iii)
Let $n$ ($n \geq 10$) be an even integer.
We define the following sets:
\begin{align*}
S  &:=\{ v_0, v_2, v_4, \ldots, v_{n-2}\}; \\
T_0 &:=\{v_0, v_{3}, v_{5}, \ldots, v_{n-5}, v_{n-3} \}; \\
T_2 &:=\{v_2, v_{5}, v_{7}, \ldots, v_{n-3}, v_{n-1} \}; \\
T_i&:=\{v_i,  v_1, \ldots,  v_{i-3}, v_{i+3}, v_{i+5}, \ldots, v_{n-1} \}
\quad (i=4, 6, \ldots,n-4); \\
T_{n-2} &:=\{v_{n-2}, v_{1}, v_{3}, \ldots, v_{n-7}, v_{n-5} \}.
\end{align*}
It is easy to see that  $S$ and $T_i$ for $i=0$, $2$, \ldots, $n-2$
are cliques of $\overline{C_n}$.
Take an edge $v_iv_j$ of $\overline{C_n}$ with $i < j$.
We consider all the possible cases in the following:
\begin{itemize}
\item{}
If both $i$ and $j$ are even,
then $v_iv_j$ is covered by $S$.
\item{}
If $i$ is even (resp.\ odd) and $j$ is odd (resp.\ even),
then $j-i \ge 3$ and $v_iv_j$ is covered by $T_i$ (resp.\ $T_j$).
\item{}
Suppose that both $i$ and $j$ are odd.
If $j-i=2$ or $4$, then $v_i$ and $v_j$ both are adjacent to $v_{j+3}$
since $n \ge 10$.
If $j-i \ge 6$, then $v_i$ and $v_j$ both are adjacent to $v_{i+3}$.
Thus there is an even integer $k$ such that
both $v_iv_k$ and $v_jv_k$ are edges of $\overline{C_n}$.
This implies that $v_iv_j$ is covered by $T_k$.
\end{itemize}
Therefore the family $\{S, T_0, T_2, T_4, \ldots, T_{n-2} \}$
is an edge clique cover of $\overline{C_n}$,
and so $\theta_e(\overline{C_n}) \le \frac{n}{2}+1$.
\end{proof}

Now by Theorems~\ref{thm:Cor2} and \ref{lem:ECC},
the following theorem immediately holds:

\begin{Thm}
Let $\overline{C_n}$ be the complement of a cycle with $n$ vertices
and $p$ be a positive integer.
\begin{itemize}
\item[{\rm (i)}]
For $p \leq 4$, $\overline{C_6}$ is a $p$-competition graph. \\
For $p \leq 1$, $\overline{C_7}$ is a $p$-competition graph. \\
For $p \leq 5$, $\overline{C_8}$ is a $p$-competition graph.
\item[{\rm (ii)}]
If $n \geq 9$ is odd and $p \leq \frac{n-3}{2}$,
then $\overline{C_n}$ is a $p$-competition graph.
\item[{\rm (iii)}]
If $n \geq 10$ is even and $p \le \frac{n}{2}$,
then $\overline{C_n}$ is a $p$-competition graph.
\end{itemize}
\end{Thm}



\section{Concluding Remarks}

In this note, we gave a necessary and sufficient condition
for a cycle $C_n$ with $n$ vertices and
a sufficient condition for
the complement of a cycle to be a $p$-competition graph.
It would be interesting to give a necessary condition for
the complement of a cycle to be a $p$-competition graph.


\end{document}